\documentclass[12pt]{article}
\usepackage{amssymb,amsfonts,amsmath,amsthm, breqn, color, float, mathtools, url, tikz, algorithmic, mathrsfs}

\usepackage[legalpaper, left=1. in,top=0.8in,right=1. in,bottom=0.8in]{geometry} 

\usetikzlibrary{arrows}

\usetikzlibrary{automata, positioning}

\usepackage{caption}
\newcommand{\one}[1]{\mbox {\bf 1}_{\{#1\}}}

\newcommand{\pp}{{\mathbb P}}

\newcommand{\smf}{{\mathfrak S}}

\newcommand{\wmf}{{\mathfrak W}}

\newcommand{\lmf}{{\mathbb L}}

\newcommand{\nn}{{\mathbb N}}
\newcommand{\ee}{{\mathbb E}}

\newcommand{\rr}{{\mathbb R}}
\newcommand{\cc}{{\mathbb C}}

\newcommand{\cala}{{\mathcal A}}

\newcommand{\cald}{{\mathcal D}}

\newcommand{\calg}{{\mathcal G}}

\newcommand{\cals}{{\mathcal S}}

\newcommand{\calb}{{\mathcal B}}

\newcommand{\calj}{{\mathcal J}}
\newcommand{\cali}{{\mathcal I}}
\newcommand{\calo}{{\mathcal O}}

\newcommand{\caly}{{\mathcal Y}}

\newcommand{\beq}{\begin{eqnarray*}}
	\newcommand{\feq}{\end{eqnarray*}}
\newcommand{\beqn}{\begin{eqnarray}}
	\newcommand{\feqn}{\end{eqnarray}}

\newtheorem{theorem}{Theorem}
\newtheorem*{conj*}{Conjecture}
\makeatletter \@addtoreset{theorem}{section}\makeatother

\newtheorem{lemma}[theorem]{Lemma}

\newtheorem*{theorema*}{Theorem~A}
\newtheorem*{theoremb*}{Theorem~B}
\newtheorem*{theoremc*}{Theorem~C}
\newtheorem*{theoremd*}{Theorem~D}
\newtheorem*{theoreme*}{Theorem~E}
\newtheorem*{theoremf*}{Theorem~F}
\newtheorem*{cld*}{Condition $\mbox{LD}_d$}
\newtheorem*{theorem*}{Theorem}
\newtheorem{proposition}[theorem]{Proposition}
\newtheorem{corollary}[theorem]{Corollary}

\newtheorem{conjecture}[theorem]{Conjecture}

\newtheorem{problem}[theorem]{Problem}

\usepackage{graphicx}
\def\BState{\State\hskip-\ALG@thistlm}
\makeatother
\DeclarePairedDelimiter\ceil{\lceil}{\rceil}
\DeclarePairedDelimiter\floor{\lfloor}{\rfloor}
\DeclareCaptionLabelFormat{noname}{#2}
\newlength\myindent
\DeclareMathOperator{\con}{\mbox{con}}
\DeclareMathOperator{\red}{\mbox{red}}

\newcommand*\pFqskip{8mu}
\catcode`,\active
\newcommand*\pFq{\begingroup
	\catcode`\,\active
	\def ,{\mskip\pFqskip\relax}%
	\dopFq
}
\catcode`\,12
\def\dopFq#1#2#3#4#5{%
	{}_{#1}{\mathscr F}_{#2}\biggl[\genfrac..{0pt}{}{#3}{#4};#5\biggr]%
	\endgroup
}

\DeclareMathOperator{\gammaf}{\Gamma}

\DeclareMathOperator{\malp}{\mathtt{K}}

\newcommand{\alp}[1]{[\malp]^{#1}}

\DeclareMathOperator{\LL}{\mathtt{A}}

\DeclareMathOperator{\HH}{\mathtt{H}}

\title{Consecutive Pattern Containment and c-Wilf Equivalence}

\author{Reza~Rastegar\thanks{Center of Excellence for Data Science and Modeling, Occidental Petroleum Corporation, Houston, TX 77046 e-mail:   reza.j.rastegar@gmail.com} }

\begin{document}
	\maketitle
	\begin{abstract}
		We offer elementary proofs for several results in consecutive pattern containment that were previously demonstrated using ideas from cluster method and analytical combinatorics. Furthermore, we establish new general bounds on the growth rates of consecutive pattern avoidance in permutations.
	\end{abstract}
	\section{Introduction}
	
	Consecutive pattern containment is a variation of the standard definition of pattern containment in permutations. In this containment, the entries of a pattern are required to appear in adjacent positions. This problem is significant in combinatorics, as it naturally arises in descents, peaks, and alternating permutations. The first systematic analysis of consecutive pattern containment is credited to Elizalde and Noy \cite{eliz8}, who defined it formally and established several fundamental results. For more information on the broader subject of pattern containment, we refer the reader to \cite{Bbook} and \cite{Kbook}. Additionally, for a more focused review of consecutive pattern containment, see \cite{eli2}. In this paper, we revisit consecutive pattern containment in both permutations and words, and we present several new and previously known facts. Our proofs are based on elementary probabilistic arguments.\par

	We begin by establishing some definitions. Let $\nn$ denote the set of natural numbers $\{1, 2, 3, \ldots\}$, and let $\nn_0$ denote the set of non-negative integers, i.e., $\nn_0 = \nn \cup \{0\}$. Given a set $\cala$, $\#\cala$ refers to the cardinality of $\cala$. The notation $\widetilde\cala$ is used to denote the complement of $\cala$ with respect to the universe set.  For a given $k \in \nn$, we denote the set $\{1, 2, \cdots, k\}$ by $[k]$, and the set of all $d$-subsets by $[k]_d$.  Denote by $\cals_n$ the set of all permutations of the set $[n]$. A word $w$ is a sequence $w := w_1\cdots w_n$, where $w_i$ represents the $i$-th entry of $w$, selected from an arbitrary set of letters referred to as the alphabet. Throughout this paper, we use $[\malp]$ as the {\it{generic}} alphabet set, where $\malp\in \nn$ is its size. We adopt the convention that $\alp 0=\{\epsilon\},$ where $\epsilon$ is an empty word. A \emph{$\malp$-ary word} of length $n$ is an element of $\alp n,$ $n\in\nn.$ A \emph{pattern} with $d$
	distinct letters is any word chosen from the alphabet set $[d]$. 
	
	Let us define an occurrence of the pattern $v$ in an arbitrary word $w$ of length $n \geq d$. Such an occurrence is represented by a sequence of $d$ indices $1 \leq j_1 < j_2 < \cdots < j_d \leq n$ such that the subsequence $w_{j_1}\cdots w_{j_d}$ is order-isomorphic to the pattern $v$. In other words,
	\beq
	w_{j_p} < w_{j_q} \iff v_p < v_q \quad \text{and} \quad w_{j_p} = w_{j_q} \iff v_p = v_q, \quad \forall\  1 \leq p,q \leq d.
	\feq
	Here, we say  $w_{j_1}\cdots w_{j_d}$ reduces to $v$ and write $\red(w_{j_1}\cdots w_{j_d})=v$. If $j_{p+1}= j_p+1$ for all $p\in [d-1]$, then we call the subsequence $w_{j_1}\cdots w_{j_d}$ a
	\emph{consecutive} occurrence of the pattern $v$. 
	For any arbitrary word $w$ we denote by $\con_v(w)$ the number of consecutive occurrences of $v$. For instance, if $v$ is the \emph{inversion} $21$ and $w=35239$, then there are three occurrences and one consecutive occurrence, those are $w_1w_3=32,$ $w_2w_3=52,$ and $w_2w_4=53$, and consequently $\con_v(w)=1$. We say that a word $w$ \emph{contains the pattern $v$ consecutively} exactly $r$ times, $r\in\nn_0,$ if $\con_v(w)=r$. For a set of words $\cala$ and $r\in \nn_0,$ we denote by $\calg_r^v(\cala)$ the set of words in $\cala$ containing $v$ consecutively exactly $r$ times. That is,
	\beqn
	\label{gfr}
	\calg_r^v(\cala)=\{w\in \cala:\con_v(w)=r\}.
	\feqn
	We use $g_r^v(\cala)$ to refer to $\#\calg_r^v( \cala)$. In the case of $r=0,$ we use the term {\it avoidance} instead of {\it occurrence}. 
	
	An important area of study in pattern containment is to determine when two patterns are avoided by the same
	number of permutations or words of the same length, or more generally when the distribution of occurrences of the two patterns is the same. We say that two patterns $v$ and $w$ are \emph{c-Wilf-equivalent in words} if $g_0^v(\alp n) = g_0^w(\alp n)$ for all
	$\malp ,n\in \nn$. Moreover, they are \emph{strongly c-Wilf-equivalent} if $g_r^v(\alp n) = g_r^w(\alp n)$ for all $\malp, n\in
	\nn$ and $r\in \nn_0$. Similarly, $v$ and $w$ are \emph{c-Wilf-equivalent in permutations} if $g_0^v(\cals_n) =
	g_0^w(\cals_n)$ for all $n\in \nn$, and are \emph{strongly c-Wilf-equivalent} if $g_r^v(\cals_n) = g_r^w(\cals_n)$ for
	all $n \in \nn$ and $r\in \nn_0$. For example, every permutation pattern $v$ is strongly c-Wilf-equivalent to its
	reversal both in words and permutations. An important open problem in consecutive patterns (whose analogue in the
	classical pattern avoidance is wide open as well) is
	
	\begin{problem} \label{prob1}
		Classify all patterns into (strongly) c-Wilf equivalence classes in permutations (in words).      
	\end{problem}
	
	In contrast to classical patterns, Nakamuta conjectured that the word ``strongly" in the above definition is unnecessary in the context of permutations (see \cite{nak1}, Conjecture 5.6). Specifically, he proposed that
	
	\begin{conjecture}[Nakamuta\rq{s} conjecture]\label{conj_nakamura}
		If two patterns are c-Wilf-equivalent in permutations, then they are also strongly c-Wilf-equivalent in permutations.    \end{conjecture}
	
	Although these problems remains unsolved, partial progress has been made towards potential solutions. To describe the progress and state the results, we first define an essential notion of overlapping concerning the structure of a
	pattern. We denote by $\calo_v$ the set of overlaps of $v$, which we define as the set of indices $i$ with $1\leq i < d$
	such that the first and the last $i$ positions have the same relative order; that is, $\red(v_{1},\cdots, v_i ) =
	\red(v_{d-i+1}, \cdots, v_{d})$. 
	Note that $1 \in \calo_v$ for every pattern $v\in \cals_d$. If $d\geq 3$, a pattern $v \in \cals_d$
	for which $\calo_v = \{1\}$ is said to be non-overlapping, also known as minimally overlapping. Equivalently, $v$ is
	non-overlapping if two occurrences of $v$ in a permutation cannot overlap in more than one position. Duane and Remmel \cite{dua1} and B\'ona \cite{bona3} studied non-overlapping patterns, and B\'ona proved that the proportion of non-overlapping patterns of any length $d$ is at least $0.364$. Additionally, it is straightforward to show that for $d\geq 3,$ the pattern $v$ is monotone, that is $v\in \{1\cdots d, d\cdots 1\}$, if and only if $\calo_v=[d]$, if and only if $d-1\in \calo_v$; see, for
	instance, Lemma 9, \cite{pera1}. Hence, we can safely consider that the maximum overlap of a non-monotone pattern $v$
	is at most $d-2$.

	A sufficient condition for strong c-Wilf-equivalence in permutations for two paterns of the same overlap set was given
	independently by Khoroshkin and Shapiro \cite{khor1}, and Nakamura \cite{nak1}. They showed for any two patterns
	$v,w\in \cals_d$, 
	\beqn \label{khor-cond}
	&& \mbox{If $\calo_v = \calo_w$ and if for all $i\in \calo_v,$ we have} \notag \\
	&& \qquad \qquad \{v_1,\cdots, v_{d-i}\} = \{w_1,\cdots, w_{d-i}\} \quad \mbox{and} \quad \{v_{i+1},\cdots, v_{d}\} = \{w_{i+1},\cdots,
	w_{d}\},
	\feqn
	then $v$ and $w$ are strongly c-Wilf-equivalent in permutations. The techniques used to prove this result were based on the Goulden-Jackson cluster method (see \cite{nak1}) and a homological duality argument applied to specific clusters (see \cite{khor1}). Our first result below provides a similar result
	that applies to words. Our proof is based on classical findings related to the observation of a predetermined sequence within an i.i.d. sequence of discrete experiments. It is more elementary than previous proofs and leads to a stronger conclusion, which can be extended to permutations using Corollary \ref{word-perm-cor}. In fact, our approach can establish a form of equivalence known as super-strong Wilf-equivalence, as defined in \cite{dwyer1}.

	\begin{theorem} \label{words_thm}
		Let $v$ and $w$ be two patterns in $\cals_d.$ Then, $v$ and $w$ are (strongly) c-Wilf-equivalent in words if \eqref{khor-cond} holds.
	\end{theorem}

	In the case of non-overlapping patterns $v$ and $w$, the condition \eqref{khor-cond} simply states that if $v_1=w_1$ and $v_d=w_d$, then $v$ and $w$ are strongly c-Wilf-equivalent in permutations. This fact has been previously shown in \cite{dot1, eliz4, dua1}. A converse of this statement for non-overlapping patterns was conjectured by Elizalde in \cite{eliz4} and proved by Lee and Sah \cite{lee1}. To formally state the result, we define a pattern $v\in \cals_d$ to be in standard form if $v_1<v_d$ and $v_1+v_d\leq d+1$. It is worth noting that for any $v\in\cals_d$, at least one permutation among $v$, its reverse, complement, and reversed-complement is in standard form. The result of Lee and Sah states that for any two standard form non-overlapping patterns $v,w\in \cals_d$, if $v$ and $w$ are c-Wild-equivalent in permutations, then $v_1=w_1$ and
	$v_d=w_d$.
	
	The condition $v_1=w_1$ and $v_d=w_d$ is clearly sufficient for non-overlapping patterns $v,w\in \cals_d$ to be strongly c-Wilf-equivalent. Therefore, the result of Lee and Sah completely characterizes (strong) c-Wilf-equivalence classes in permutations for non-overlapping patterns. See Proposition~\ref{non-overlapping-wilf-thm} below. Although this result is only applicable to non-overlapping patterns, Dwyer and Elizalde formulated another statement without this restriction and showed that it is equivalent to the theorem of Lee and Sah. More details on this can be found in \cite{dwyer1}. Their more general statement is that for any patterns $v,w\in \cals_d$ in standard form, if $v$ and $w$ are c-Wilf-equivalent, then $v_1=w_1$ and $v_d=w_d$. We remark that the inverse of this statement does not hold as one can find counterexamples. More specifically there are known examples of patterns $v$  and $w$ satisfying $v_1=w_1$ and $v_d=w_d$ that are not c-Wilf-equivalent in permutations. See \cite{dwyer1} for more information and several examples.

	Below, we present several straightforward observations that we think could help in analyzing consecutive pattern containment for non-overlapping patterns. To explain these findings, we define
	
	\beqn \label{a_h_def}
	a^{(r+1)}_n(v):=\frac{g^v_{r}(\cals_{n+d-1})}{(n+d-1)!} \quad \mbox{and} \quad h^{(r+1)}_{\malp,n}(v):=\frac{g^v_{r}(\alp{n+d-1})}{\malp^{n+d-1}}.
	\feqn
	
	The first observation establishes a linear recursive equation for $a^{(r)}_n(v)$ and $h^{(r)}_{\malp,n}(v)$. The arguments are based on repetitive exclusion-inclusion of probability events.
	
	\begin{theorem} \label{non-overlapping-bound-thm}
		\begin{itemize}
			\item[(a)] Let $v\in \cals_d$ be a non-overlapping  pattern. For a given $n \geq d,$ set $m:=\floor{\frac{n-1}{d-1}}$. Then
			\beqn \label{qqq1}
			a_n^{(1)}(v) - a_{n-1}^{(1)}(v) - \frac{1}{d!} a^{(1)}_{n-d}(v) - \sum_{j=1}^{m} (-1)^{j} \LL_j(v) a^{(1)}_{n-jd+j-d}(v) = 0,
			\feqn
			where $\LL_j(v)$s are constants defined by \eqref{Ln_value}. 
			\item[(b)] Fix $\malp \in \nn$. Let $v\in \cals_d$ be any non-overlapping permutation pattern with $v_1<v_d$. Then, there is $n_0\in \nn$ such that for any $n\geq n_0$ 
			\beqn \label{www1}
			h_{\malp,n}^{(1)}(v) - h_{\malp,n-1}^{(1)}(v) - \frac{1}{\malp^d} h^{(1)}_{\malp,n-d}(v) - \sum_{j=1}^{\malp} (-1)^{j} \HH_{\malp, j}(v) h^{(1)}_{\malp, n-jd+j-d}(v) = 0,
			\feqn
			where $\HH_{\malp, j}(v)$ is defined by \eqref{Ln_value_word}.
		\end{itemize}
	\end{theorem}
	
	The proof of the first part is given in Section \ref{perm_linear}. As for the second part, the proof is very similar to the proof of the first part and we sketch it in the beginning of Section \ref{word_sec}.
	
	As a consequence of this result, we present the following proposition that had been previously established for permutations by Elizalde and Noy \cite{eliz9}, utilizing alternative techniques such as cluster method and analytical combinatorics. In contrast, our proof relies on elementary probability arguments and is presented for permutations in Section \ref{perm_linear}.
	
	\begin{proposition} \label{non-overlapping-wilf-thm} 
		Suppose $w,v\in \cals_d$ are two non-overlapping patterns that are c-Wilf-equivalent in permutations (resp. words). Then, they are also strongly c-Wilf-equivalent in permutations (resp. words). 
	\end{proposition}
	
	It is noteworthy that the proof for words is identical and is omitted. 
	
	Next, following the approach of Theorem \ref{non-overlapping-bound-thm}, we express a recursive equation for monotone patterns. This can be formulated as follows:
	
	\begin{theorem} \label{non-monotone-bound-thm}
		Let $v\in \cals_d$ be a monotone pattern, that is $v=1\cdots d$ or $v=d\cdots 1$.  For $k\geq 2$ and $n\geq
		k(d-1),$ set $m:=\floor{\frac{n-1}{d}}$. Then,
		\beq
		a_{n-1}^{(1)}(v) - a_n^{(1)}(v)  + \sum_{k=1}^m (-1)^{k} M_k a_{n-kd-1}^{(1)}(v) + \widetilde M_m= 0,
		\feq
		where $M_k$ and $\widetilde M_m$ are defined by \eqref{M_k_def1}, \eqref{M_k_def2}, and \eqref{M_k_def3}.
	\end{theorem}
	
	Obtaining equations similar to those demonstrated by Theorem \ref{non-overlapping-bound-thm} and Theorem \ref{non-monotone-bound-thm} is often challenging, if not impossible, in most cases. However, we can derive the following inequality applicable to any permutation pattern.
	
	\begin{theorem} \label{war-thm}
		Let $v$ be any pattern in $S_d$. Fix $\ell\geq 2$. Then, for $n>\ell+2d$ we have
		\beq
		0\leq a_{n}^{(1)}(v) -  a_{n-1}^{(1)}(v) + \beta_{\ell-1}(v) a_{n-\ell-d+1}^{(1)}(v) \leq \frac{d-1}{d!} a_{n-\ell-2d+2}^{(1)}(v),        \feq
		where 
		\beqn \label{beta_ell}
		\beta_{\ell-1}(v) := a_{\ell-1}^{(1)}(v) - a_{\ell}^{(1)}(v) =  \frac{g^v_{0}(\cals_{\ell+d-2})}{(\ell+d-2)!} - \frac{g^v_{0}(\cals_{\ell+d-1})}{(\ell+d-1)!} .
		\feqn
	\end{theorem}
	
	A proof of this inequality is presented in the Section \ref{growth-bound-sec}. It is worth noting that this inequality can be modified to hold for words. More specifically, suppose $\ell\geq 2$ and $v\in \cals_d$. Then, for $n>\ell+2d$ we have
	\beq
	0\leq h_{\malp,n}^{(1)}(v) -  h_{\malp,n-1}^{(1)}(v) + \beta_{\ell-1}(v) h_{\malp,n-\ell-d+1}^{(1)}(v) \leq \frac{d-1}{\malp^d} h_{\malp,n-\ell-2d+2}^{(1)}(v),        
	\feq
	where 
	\beq
	\beta_{\ell-1}(v) :=h_{\malp, \ell-1}^{(1)}(v) - h_{\malp, \ell}^{(1)}(v) =  \frac{g^v_{0}(\alp{\ell+d-2})}{\malp^{\ell+d-2}} - \frac{g^v_{0}(\alp{\ell+d-1})}{\malp^{\ell+d-1}} .
	\feq
	
	The proof is nearly identical to that of Theorem \ref{war-thm} and is therefore omitted.
	
	Having obtained the inequality of Theorem~\ref{war-thm}, our attention now turns to the growth rate of consecutive pattern avoidance in permutations. 
	
	Ehrenborg et al. \cite{kitaev1} previously studied the growth rate of $g_0^v(\cals_n)$, proving Warlimont's conjecture \cite{war1} on the asymptotic behavior of the number of permutations that avoid consecutive patterns using the spectral theory of integral operators on $L_2([0,1]^{d-1})$. Specifically, they demonstrated that the asymptotic value of $g_0^v(\mathcal{S}_n)$ is given by $g_0^v(\mathcal{S}_n) =  \left(b_v \rho_v^{n} + O(r_v^n)\right)n!$, where $b_v$, $\rho_v$, and $r_v$ are positive constants that depend solely on $v$, with $\rho_v>r_v$. While it is not the main focus of our investigation, it would be worthwhile to obtain a similar asymptotic statement for $g_r^v(\cals_n)$ for any other fixed value of $r \in \nn$, and to determine whether any of the constants involved in the $r=0$ case are present in the asymptotic form of the general case for $r$. Notably, if $r\geq n-d+2$ and all elements of $\cala$ have lengths of at most $n$, then $g_r^v(\cala)=0$, as any word in $\cala$ has at most $n-d+1$ consecutive subwords of length $d$ that could potentially be order-isomorphic to $v$. Furthermore, for the case of $\cala=\cals_n$, we have the following lemma that, to the best of our knowledge, is new in the context of consecutive pattern containment:
	\begin{lemma}
		For any $v\in \cals_d$, and any fixed $r\in \nn_0$, the limit $\lim_n \left( \frac{g_r^v(\cals_n)}{n!} \right)^{\frac{1}{n}}$ exists and is equal to $\rho_v$.
	\end{lemma}
	The proof is similar to that of Theorem B, \cite{reza2} and hence it is omitted. A similar result can be shown for the words by conducting a modified version of Theorem 1.1-(a), \cite{reza1}. 
	
	In our next result, we provide further insights into the growth rate parameter $\rho_v$ in permutations. One major result in consecutive pattern containment posits that, among all patterns of fixed length $d$, the increasing pattern $12\cdots d$ (and the decreasing pattern $d\cdots 1$) is the one for which the number of permutations avoiding it is asymptotically the most significant. In particular, Elizalde (see \cite{eliz4}) proved that, for any pattern $v\in\mathcal{S}_d$,
	\beq
	\rho_{1\cdots(d-2)d(d-1)} \leq \rho_v \leq \rho_{12\cdots (d-1)d},
	\feq
	and for non-overlapping pattern $v\in\mathcal{S}_d$,
	\beq
	\rho_{1\cdots(d-2)d(d-1)} \leq \rho_v \leq \rho_{134\cdots d2}.
	\feq
	
	Furthermore, the study by Perarnau \cite{pera1} employed various probabilistic techniques to establish that for any $d\in\nn$ and non-monotone pattern $v\in\mathcal{S}_d$, the asymptotic growth rate of the number of permutations in $\mathcal{S}_n$ avoiding the consecutive pattern $v$ is bounded above by $1-\frac{1}{d!} + O(\frac{1}{d^2 d!})$, while for any $v\in\mathcal{S}_d$, it is bounded below by $1-\frac{1}{d!} - O(\frac{d-1}{ d!^2})$. Our result augments theirs. 
	
	\newpage 
	\begin{theorem} \label{growth_theorem}
		Let $v$ be any pattern in $S_d$.  
		\begin{itemize}
			\item[(a)] Then, for any $\alpha \geq 2$, and $\lambda := 1+d(\alpha-1),$
			\beqn \label{mineq}
			\left|a_{\alpha kd}^{(1)}(v) - \left(a_{\lambda}^{(1)}(v)\right)^{k}\right| \leq  (k+2) \left(a_{\lambda}^{(1)}(v)\right)^{k-1}, \qquad \forall k\in \nn.
			\feqn
			In particular, this implies
			\beqn
			\rho_v \leq  \left( \frac{g^v_{0}(\cals_{\alpha d})}{(\alpha
				d)!}\right)^{\frac{1}{\alpha d}}. \label{rho_bound}
			\feqn
			\item[(b)] For any $\ell\in \nn$, larger than $1$, 
			\beq
			\rho_{v,l} \leq \rho_v \leq \rho_{v, u},
			\feq
			where $\rho_{v,l}$ and $\rho_{v,u}$ are the dominant roots of the equations
			\beqn \label{poly1}
			\rho^{\ell+d-1} - \rho^{\ell+d-2} + \beta_{\ell-1}(v) = 0
			\feqn
			and
			\beqn \label{poly2}
			\rho^{\ell+2d-2} - \rho^{\ell+2d-3} + \beta_{\ell-1}(v) \rho^{d-1} -\frac{d-1}{d!} = 0,
			\feqn
			where $\beta_{\ell-1}(v)$ is defined by \eqref{beta_ell}.
			In particular, choosing $\ell$ such that $\beta_{\ell-1}(v) \neq \frac{(m-1)^{m-1}}{m^m}$ with $m:=d+\ell-1$, there exists $\gamma\in (1,\infty)$ such that 
			\beq
			1-\beta_{\ell-1}(v) - \left(\beta_{\ell-1}(v)\right)^{\gamma}  \leq \rho_v.
			\feq
		\end{itemize}
	\end{theorem}

	One immediate consequence of part (b) is that $\lim_d \min_{v\in \cals_d}\rho_v = 1$ as one would expect. This is evident since the dominant roots of the equations \eqref{poly1} and \eqref{poly2} converge to $1$ as $d$ grows to infinity. 
		
	We would like to highlight the distinction between Perarnau's approach and ours. While Perarnau's bounds remain same across all non-monotone patterns, our methodology introduces a novel flexibility. By allowing the parameter $\alpha$ and $\ell$ to vary, we offer a spectrum of upper bounds tailored to specific patterns. This dynamic nature of our bounds empowers the reader to adapt her analyses according to the intricacies of each pattern under consideration. Furthermore, our approach enables a direct enumeration of \eqref{beta_ell} or the right-hand side of equation \eqref{rho_bound}, facilitating a more streamlined process for deriving bounds. 
	
	Finally, we note that, for the sake of simplifying the presentation and proof of our results, we restricted our consideration to permutation patterns consistently within $\cals_d$. We maintain that this restriction is purely technical, and that all findings pertaining to words, notably Theorems \ref{words_thm}, \ref{non-overlapping-bound-thm}-(b), and \ref{non-overlapping-wilf-thm}, remain applicable to any $d$-ary pattern within the framework of pattern containment in words.
	
	\section{Consecutive containment in permutations} \label{perm_section}
	
	In order to prove our results, we found it easier to reformulate the problem of consecutive pattern containment within a probabilistic framework. However, it is also possible to present combinatorial proofs of many of these results without relying on the language of probability. Let $(Y_n)$ be a sequence of independently and identically distributed random variables $(Y_n)_{n\in \nn}$, uniformly sampled from $[0,1]$. We will denote by $\pp$ the probability law in a suitable probability
	space that supports all random variables considered in this section. Consider a discrete-time Markov process $(X_n)_{n\in \nn}$ on the space $[0,1]^d$, where $X_n:=Y_{n}\cdots
	Y_{n+d-1}$. The transition probability of $X_n$ is defined as
	\beq
	P(x, \cala) = \lmf(\{y\ | (x_2,\cdots, x_d, y)\in \cala\}), \quad x:=x_1\cdots x_d\in [0,1]^d, \quad \cala\subset [0,1]^d,
	\feq
	where $\lmf$ represents the Lebesgue measure on $[0,1]$. It is well-known that $\red(Y_1 \cdots Y_n)$ follows the same distribution as a uniformly randomly sampled permutation from $\cals_n$. This implies that
	\beq
	a_n^{(r)}(v)=\frac{g_{r-1}^v(\cals_{n+d-1})}{(n+d-1)!} = \pp(\con_v(\red(Y_1\cdots Y_{n+d-1}))=r-1).
	\feq
	In the rest of this section, we will use $\calb_i(v)$ to denote the event $\{\red(X_i)\neq v\}$ and use $\widetilde \calb_i(v)$ to refer to its complement. We define $\calb_{i}^j(v)$ as $\cap_{s=i}^j \calb_s(v)$ for $i\leq j$. When $i>j$, we define $\calb_{i}^j(v)$ as the universe set with probability measure of one. We will drop $v$ whenever the context is clear. The following facts will be used frequently without explicit mention: For $i\in \nn$,
	\beq
	\pp(\widetilde \calb_i) = \frac{1}{d!} \quad \text{and} \quad \pp(\calb_i) = 1-\frac{1}{d!},
	\feq
	and for $i\geq 2$,
	\beq
	\pp(\calb_{1}^{i-1}, \widetilde \calb_{i}) = \pp(\calb_{1}^{i-1}) - \pp(\calb_{1}^{i}).
	\feq
	
	We commence by presenting the following observation, which although not employed directly in our proofs, serves as a straightforward proof of the existence of the growth rate $\rho_v$ and the inequality $\rho_v<1$. This outcome was originally demonstrated in \cite{eli1} (see Theorem 4.1 herein). It is noteworthy to mention that Ehrenborg et al. \cite{kitaev1} established a more refined outcome, as mentioned in the introduction.
	
	Let $\smf_r^v(x)$ be the exponential generating function of $g_r^v(\cals_n)$; that is,
	\beq
	&&\smf_r^v(x) = \sum_{n=d}^\infty g_r^v(\cals_n)\frac{x^n}{n!}, \quad r\in \nn_0.
	\feq

	\begin{lemma}
		Let $v\in\cals_d$ be a given pattern. Then, there exists a constant $\kappa_v>1$ such that $\smf_0^v(\kappa_v)<\infty$. In particular, this implies that $\rho_v = \lim_n (a_n^{(1)}(v))^{1/n}$ exists and is at most $\frac{1}{\kappa_v}<1$.
	\end{lemma}
	\begin{proof}
		Our initial task is to verify Doeblin's condition for $X_n$ (refer to \cite{tw1} for a detailed discussion). To do so,  let $P^k(.,.)$ denote the $k$-th step transition probability of the Markov chain $(X_n)$. For every $a\in (0,1)$, for any $x\in [0,1]^d$, and any Borel measurable subset $\cala$ of $[0,1]^d$, we have the following inequality:
		\beq
		(1-a) \sum_{k=0}^\infty a^k P^k(x,\cala) &\geq& (1-a) \sum_{k=d}^\infty a^k P^k(x,\cala) \geq (1-a) \sum_{k=d}^\infty a^k
		\lmf(\cala) \\
		&=& a^d \lmf(\cala).
		\feq
		By  Theorem 16.0.2 of \cite{tw1}, this implies that $X_n$ is irreducible and has at most one invariant measure. In the next line, we show that this invariant measure is, in fact, the Lebesgue measure on $[0,1]^d$:
		\beq
		P\lmf(\cala):=\int_{[0,1]^d} P(x,\cala) dx = \int_0^1\cdots \int_0^1 \lmf({(x_2,\cdots,x_d,y)\in \cala}) dx_2\cdots dx_d dy =
		\lmf(\cala),
		\feq
		where $\lmf$ is the Lebesgue measure on $[0,1]^d$. \par
		
		Let $T$ denote the first time that $(X_n)$ visits $\cala_v:=\{x\in [0,1]^d\ | \red(x)=v\}.$ Since Doeblin's condition holds for the aperiodic chain $X_n$, Theorem 16.0.2 of \cite{tw1} also implies the existence of $\kappa_v > 1$ such that
		
		\begin{equation}\label{sup_ineq}
			\sup_{x\in [0,1]^d} \ee(\kappa_v^T|X_1=x) < \infty,
		\end{equation}
		where $\ee$ denotes the expectation operator with respect to the probability measure $\pp$. The inequality \eqref{sup_ineq} implies that $\ee(\kappa_v^T)$ is finite. To proceed, we can write:
		\beq
		&&\ee(\kappa_v^T) = \sum_{n=1}^\infty \pp(T=n)\kappa_v^n = \sum_{n=1}^\infty \left(\pp(T>n-1) - \pp(T>n)\right) \kappa_v^n \\
		&& \quad =  1 + (\kappa_v-1)\sum_{n=1}^\infty \pp(T>n) \kappa_v^n = 1 + (\kappa_v-1)\smf _0^v(\kappa_v).
		\feq 
		We have now completed the proof of the first part. To prove the second part, we can simply apply Markov's inequality:
		\beq
		\pp(T>n\ |\ X_1=x) = \pp(\kappa_v^{T}>\kappa_v^n \ | \ X_1=x) \leq \frac{ \ee(\kappa_v^{T}\ |\ X_1=x)}{\kappa^n}.
		\feq
		Therefore, we can write:
		\beq
		a_n^{(1)}(v) = \pp(T>n) \leq \int_{[0,1]^d}\frac{ \ee(\kappa_v^{T}\ | \ X_1=x)}{\kappa_v^n}dx \leq \left(\sup_{x\in[0,1]^d}
		\ee(\kappa_v^{T}\ | \ X_1=x)\right)\kappa_v^{-n}.
		\feq
		This implies that:
		\beq
		\rho_v = \lim_n (a_n^{(1)}(v))^{\frac{1}{n}} \leq \frac{1}{\kappa_v} < 1,
		\feq
		which completes the proof of the second part.
	\end{proof}
	
	\subsection{Linear recursions and c-Wilf-equivalence} \label{perm_linear}
	
	In this section, we will utilize repetitive exclusion-inclusion arguments to derive linear relations and deduce information about c-Wilf-equivalence among patterns. The first main result we will present is the
	
	\begin{proof} [Proof of Theorem~\ref{non-overlapping-bound-thm}-(a)] 
		One key observation is that when $v$ is non-overlapping, $\widetilde\calb_s$ implies $\calb^{s+d-2}_{s+1}$ for any $s\in \nn$. We can use this fact to express $a^{(1)}_n(v)$ as a linear recursion. Specifically, for all $n\geq d$, we have:
		\beq
		&& a_{n-1}^{(1)}(v) - a_n^{(1)}(v) = \pp(\calb_2^n) - \pp(\calb_1^n) =  \pp(\widetilde \calb_1, \calb_d^n) 
		\feq
		If $n\geq 2d-1$, we can repeat the same argument for $\calb_d$ to obtain:
		\beq
		a_{n-1}^{(1)}(v) - a_n^{(1)}(v) = \pp(\widetilde \calb_1, \calb_{d+1}^{n}) - \pp(\widetilde \calb_1, \widetilde \calb_d, \calb_{2d-1}^{n})
		\feq
		By applying this argument several times, we obtain:
		\beq
		a_{n-1}^{(1)}(v) - a_n^{(1)}(v) =  \sum_{k=0}^{m} (-1)^{k} \pp(\widetilde \calb_1, \cdots, \widetilde \calb_{k(d-1)+1}, \calb_{k(d-1)+d+1}^n),
		\feq
		where $m:=\floor{\frac{n-1}{d-1}}$.
		
		We now analyze the terms in the summation from the previous paragraph. For any $0\leq k\leq m,$ we note that $\calb_{kd-k+d+1}^n$ is independent of all the other events $\widetilde \calb_1, \cdots, \widetilde \calb_{k(d-1)+1}$, so their joint probability can be expressed as a product. Let $\LL_0(v)=\frac{1}{d!}$ and recall that, by stationarity, $\pp(\calb_{kd-k+d+1}^n) = a^{(1)}_{n-kd+k-d}(v)$. Therefore, for $k\in \nn$, we can write
		\beq
		\pp(\widetilde \calb_1, \cdots, \widetilde \calb_{k(d-1)+1}, \calb_{kd-k+d+1}^n) = \LL_k(v) a^{(1)}_{n-kd+k-d}(v),
		\feq
		where $\LL_k(v) := \pp(\widetilde \calb_1, \cdots, \widetilde \calb_{k(d-1)+1})$. This gives the recursion \eqref{qqq1}. 
		
		To compute $\LL_k(v)$, we can proceed as follows. Let $$z:=(z_1\cdots, z_d, \cdots, z_{2d-1}, \cdots,
		z_{kd - (k-1)}, \cdots, z_{d(k+1)-k})$$
		be a permutation in $\cals_{(k+1)d-k}$ that is obtained from the standard reduction of an instance of the random
		sequence $Y_1\cdots Y_{d(k+1)-k}$ satisfying $\widetilde \calb_1\cap \cdots \cap \widetilde \calb_{k(d-1)+1}$ to a permutation. Since the number of choices for $z$ is $((k+1)d-k)! \LL_{k}(v)$, we can obtain the closed formula for $\LL_k(v)$ through an enumeration of $z$. We set $x_i:=z_{id+1}$.
		
		The condition $\widetilde \calb_1\cap \cdots \cap \widetilde \calb_{k(d-1)+1}$ dictates the following conditions on $(x_i)_{0\leq i \leq k+1}$ values:
		\begin{itemize}
			\item Since $v_1<v_d,$ we must have $x_{0}<x_{1}<\cdots<x_{k+1}$.
			\item Out of $(k+1)d-k$ elements in the vector $z$, at least $v_1-1$ of them are smaller than $x_0$ and at most $(d-v_1)(k+1)$ of them are larger than $x_0$. Hence,
			\beq
			v_1\leq x_0 \leq (k+1)v_1 - k.
			\feq
			\item Knowing $x_0$, the value of $x_{k+1}$ ranges between $x_0+(k+1)(v_d-1)$ and $(k+1)d-k-(d-v_d)=k(d-1)+v_d$.
			\item Knowing $x_0$ and $x_{k+1}$, for any fixed $1\leq i \leq k$, the value of $x_i$ ranges between $x_{i-1}+v_d-v_1$ and $x_{k+1} - (k+1-i)(d-v_1)$.
		\end{itemize}    
		
		We define the $i$-th block of $z$ as $z_{(i-1)d+2}\cdots z_{id}$, where $1\leq i\leq k+1$. To enumerate $z$ by parametrizing with respect to $x_0<x_1<\cdots<x_k<x_{k+1}$, we need to count the number of choices for each block of $z$ given that $x_{i-1}=z_{(i-1)d+1}$ and $x_i=z_{id+1}$ are fixed.  Then, we can use the following observations to count the number of choices for each block:
		\begin{itemize}
			\item There are exactly $v_d-v_1-1$ values in the $i$-th block that are between $x_{i-1}$ and $x_i$. Therefore, the number of possible choices for the $i$-th block is $\binom{x_i-x_{i-1}-1}{v_d-v_1-1}$.
			\item There are only $(i+1)d+(k-i)v_1-k-x_i$ values that can be assigned to $d-v_d-1$ elements of the $i$-th block that are larger than $x_i=z_{id+1}$. Therefore, the number of choices for these elements is $\binom{(i+1)d+(k-i)v_1-k-x_i}{d-v_d-1}$.
			\item Similarly, there are only $x_{i-1}-(v_d-1)i-1$ values that can be assigned to $v_1-1$ elements of the $i$-th block that are smaller than $x_{i-1}=z_{(i-1)d+1}$. Therefore, the number of choices for these elements is $\binom{x_{i-1}-(v_d-1)i-1}{v_1-1}$.
		\end{itemize}

		By collecting all the observations we have made so far, we obtain the following expression:
		\beqn \label{Ln_value}
		&& ((k+1)d-k)! \LL_{k}(v) = \sum_{x_0=v_1}^{(k+1)v_1-k}\ \  \sum_{x_{k+1}=x_0+(k+1)(v_d-1)}^{k(d-1)+v_d} \ \  \sum_{x_1=x_0+v_d-v_1}^{x_{k+1}-k(d-v_1)} \cdots \\
		&& \qquad \sum_{x_k=x_{k-1}+v_d-v_1}^{x_{k+1}-(d-v_1)} \prod_{i=1}^{k+1}\binom{x_{i} -
			x_{i-1}-1}{v_d-v_1-1}  \binom{(i+1)d+(k-i)v_1-k-x_i}{d-v_d-1} \binom{x_{i-1}-(v_d-1)i-1}{v_1-1}. \notag
		\feqn
		Hence, by combining \eqref{Ln_value} with \eqref{qqq1}, we can conclude that the proof is complete.
	\end{proof}

	We will continue with a similar line of argument and present the
	
	\begin{proof} [Proof of Proposition~\ref{non-overlapping-wilf-thm}]
		Firstly, we observe that for any $k\in\nn $ and $n\geq k$, the term $a_n^{(k)}(v)$ can be expressed as
		\beq
		a^{(k)}_n(v):=\pp(\cup_{1\leq i_1<\cdots<i_{k-1}\leq n} \calb_{1}^{i_1-1}\cap \widetilde \calb_{i_1}\cap \calb_{i_1+1}^{i_2-1}\cap\cdots \cap \widetilde \calb_{i_{k-1}}\cap \calb_{i_{k-1}+1}^{n}),
		\feq
		where $a_0^{(k)} = 1.$  Next, for $n\geq 2$, we can write 
		\beqn \label{sum_c}
		a^{(k)}_{n-1}(v) - a^{(k)}_{n}(v) = \sum_{1\leq i_1 < \cdots < i_{k-2}\leq n-1} \pp(\calb_{1}^{i_1-1}, \widetilde \calb_{i_1}, \calb_{i_1+1}^{i_2-1}, \cdots, \widetilde \calb_{i_{k-2}}, \calb_{i_{k-2}+1}^{n-1}, \widetilde \calb_{n}).
		\feqn
		We denote the term in the above summation by $c_{i_1,\cdots,i_{k-2}}^{(k-1)}$. Since $v$ is non-overlapping, i.e. $\calo_v=\{1\},$ then $\widetilde \calb_j$ implies that the event $\calb_{j+1}^{j+d-2}$ holds
		true for any $j$. Therefore, we must have $i_{s}-i_{s-1}\geq d-1$ for all $2\leq s \leq k-2$. If $i_{s}-i_{s-1}>d$ for some $2\leq s \leq k-2$, then by the independence of events that are separated by more than $d-1$ variables, we can write  $\pp(\widetilde \calb_{i_1}, \cdots, \widetilde \calb_{i_{k-2}})$
		as $\pp(\widetilde \calb_{i_{1}}, \cdots, \widetilde \calb_{i_{s-1}})\pp(\widetilde \calb_{i_{s}}, \cdots, \widetilde \calb_{i_{k-2}})$. The
		arrangement of these products is independent of the shape of $v$ and we only used the fact that $v$ is
		non-overlapping. After applying the argument repeatedly, we obtain that $c_{i_1,\cdots, i_{k-2}}^{(k-1)}$ is a product of $\LL_k(v)$ terms defined by \eqref{Ln_value}.
		
		Now, suppose that $v$ and $w$ are non-overlapping and c-Wilf-equivalent. Therefore, by Theorem \ref{non-overlapping-bound-thm}, we have $\LL_k(v) = \LL_k(w)$ for all $k\geq 2$, and consequently, $c^{(k-1)}_{i_1,\cdots, i_{k-1}}(v) = c^{(k-1)}_{i_1,\cdots, i_{k-1}}(w)$ for all $k\geq 2$. The proof follows from this fact and an obvious inductive argument based on \eqref{sum_c}.
	\end{proof}

	We now turn to the 
	
	\begin{proof} [Proof of Theorem~\ref{non-monotone-bound-thm}] 
		To prove the theorem, we will show that it holds true for $v=12\cdots d$. This same reasoning applies to $v=d\cdots 1$ due to symmetry. It is important to note that for any $s\in \mathbb{N}$, the event $\widetilde \calb_s \cap \calb_{s+1}$ is equivalent to $\widetilde \calb_s \cap \{Y_{s+d}<Y_{s+d-1}\}$. Furthermore, this event guarantees the occurrence of the event $\calb_{s+2}^{s+d-1}$.  By utilizing this simple observation, we can write the following expression for $n>d$:
		\beq
		&& a_{n-1}^{(1)}(v)-a_n^{(1)}(v) = \pp(\calb_2, \calb_3^n) -  \pp(\calb_1, \calb_2, \calb_3^n) \\
		&& \quad = \pp(\widetilde \calb_1, \calb_2, \calb_3^n) = \pp(\widetilde \calb_1, Y_{d+1}<Y_{d}, \calb_{d+1}^n).
		\feq 
		We will now apply the same trick we used for $\calb_1$ to $\calb_{d+1}$. As a result, the right-hand side of the aforementioned equation can be expressed as:
		\beq
		&&  \pp(\widetilde \calb_1, Y_{d+1}<Y_{d}, \calb_{d+2}^n) - \pp(\widetilde \calb_1,
		Y_{d+1}<Y_{d}, \widetilde \calb_{d+1}, \calb_{d+2}, \calb_{d+3}^n) \notag \\
		&& \quad = \pp(\widetilde \calb_1, Y_{d+1}<Y_{d}) \pp(\calb_{d+2}^n) - \pp(\widetilde \calb_1, Y_{d+1}<Y_{d}, \widetilde \calb_{d+1}, Y_{2d+1}<Y_{2d},
		\calb_{2d+1}^n)
		\feq
		By repeating this line of reasoning, we can arrive at the following expression:
		\beq
		&&  a_{n-1}^{(1)}(v)-a_n^{(1)}(v) = \sum_{k=1}^{m} (-1)^{k-1}\pp(\widetilde \calb_1, Y_{d+1}<Y_{d}, \widetilde \calb_{d+1},\cdots, \widetilde \calb_{(k-1)d+1},
		Y_{kd+1}<Y_{kd}) \pp(\calb_{kd+2}^n) \notag \\
		&& \qquad + (-1)^{m}\pp(\widetilde \calb_1, Y_{d+1}<Y_{d}, \widetilde \calb_{d+1},\cdots, Y_{md+1}<Y_{md}, \widetilde \calb_{md+1}),
		\feq
		where $m=\floor{\frac{n-1}{d}}$. Using stationarity, we can express the equation as follows:
		\beq
		&&a_{n}^{(1)}(v)-a_{n-1}^{(1)}(v) + \sum_{k=1}^m (-1)^{k-1} M_k a_{n-kd-1}^{(1)}(v)  + (-1)^m \widetilde M_m = 0,
		\feq
		Here, for $1\leq k \leq m$, we define $M_k$ as:
		\beq
		M_k := \pp(\widetilde \calb_1, Y_{d+1}<Y_{d}, \widetilde \calb_{d+1},\cdots, \widetilde \calb_{(k-1)d+1}, Y_{kd+1}<Y_{kd}).
		\feq
		Additionally, we define $\widetilde M_m$ as:
		\beq
		\widetilde M_m := \pp(\widetilde \calb_1, Y_{d+1}<Y_{d}, \widetilde \calb_{d+1},\cdots,  Y_{md+1}<Y_{md}, \widetilde \calb_{md+1}).
		\feq
		To proceed, we will derive a closed-form expression for $M_k$ in terms of $k$. Consider a permutation $z=(z_1,\ldots,z_{kd+1})$ obtained from the standard reduction of an instance of the random sequence $Y_1,\ldots,Y_{kd+1}$ satisfying $\widetilde \calb_1, Y_{d+1}<Y_{d}, \cdots, \widetilde \calb_{(k-1)d+1}, Y_{kd+1}<Y_{kd}$ to a permutation. Let $y_i=z_{id}$ and $x_i=z_{id+1}$, so that the pairs $(y_i,x_i)$ correspond to adjacent entries in the permutation. The conditions $\widetilde \calb_1\cap\{Y_{d+1}<Y_d\}\cap \cdots \cap \widetilde \calb_{k(d-1)+1}\cap\{Y_{kd+1}<Y_{kd}\}$ impose the following constraints on the values  $(z_i)_{1\leq i \leq
			k(d-1)+d}$:
		\begin{itemize}
			\item $y_{i}<x_{i}$.
			\item $x_i<y_{i+1}$.
		\end{itemize}
		With these properties in mind, we can count the number of possible sequences $(z_i)_{1\leq i \leq kd+1}$  by parametrizing them with respect to $y_{1}<x_1<y_2<\cdots<y_k<x_k$. This count is equal to $(kd+1)! M_{k}.$  For $k=1$ we have
		\beqn  \label{M_k_def1}
		(d+1)! M_1 = \sum_{y_1=d}^{d+1} (y_1-1) \binom{y_1-1}{d-1}.
		\feqn
		For $k\geq 2$, we have
		\beqn  \label{M_k_def2}
		&& (kd+1)! M_k = \sum_{y_1=d}^{d+1} \binom{y_1-1}{d-1} \sum_{x_1=1}^{y_1-1} \sum_{y_2=x_1+d-2}^{d+1} \binom{y_2-x_1-1}{d-2} \sum_{x_2=1}^{y_2-1} \times \\
		&& \qquad \times \cdots\sum_{y_k=x_{k-1}+d-2}^{d+1} \binom{y_k-x_{k-1}-1}{d-2} \sum_{x_k=1}^{y_k-1}1. \notag 
		\feqn    
		Similarly,
		\beqn  \label{M_k_def3}
		&& (md+1)! \widetilde M_m = \sum_{y_1=d}^{d+1} \binom{y_1-1}{d-1} \sum_{x_1=1}^{y_1-1} \sum_{y_2=x_1+d-2}^{d+1} \binom{y_2-x_1-1}{d-2} \sum_{x_2=1}^{y_2-1} \times \\
		&& \qquad \times \cdots\sum_{x_m=1}^{y_m-1}\sum_{y_m=x_{m-1}+d-2}^{d+1} \binom{y_m-x_{m-1}-1}{d-2}.\notag
		\feqn
		This completes the proof.    
	\end{proof}

	\subsection{Bounds on the growth} \label{growth-bound-sec}
	As mentioned in the introduction, obtaining a linear recursion for the general pattern $v$ appears to be challenging, if not impossible. However, we can still derive linear recursion inequalities to obtain bounds.
	
	\begin{proof}[Proof of Theorem~\ref{war-thm}-(a)]
		Let $a_n:=\pp(\calb_1^n)$ and $\beta_{n-1} := \pp(\calb_1^{n-1}, \widetilde \calb_n)$. We can write the following equation:
		
		\beqn \label{main_eq}
		a_n = \pp(\calb_1^n) = \pp(\calb_1^{n-1}) - \pp(\calb_1^{n-1}, \widetilde \calb_n) = a_{n-1} - \beta_{n-1}, \qquad n\geq 2.
		\feqn
		The initial condition is $a_1=1-\frac{1}{d!}.$ 
		To turn this equation into an inequality with homogeneous linear recurrence relations, we aim to bound the term $\beta_{n-1}$ with a linear sum of terms $a_{i}$s. In pursuit of this objective, for any $\ell\geq 2$ and any $n\geq \ell+d$, we write
		\beq
		\beta_{n-1} &=& \pp( \calb_1^{n-1}, \widetilde \calb_n ) = \pp\left( \calb_1^{n-\ell}, \calb_{n-\ell+1}^{n-1}, \widetilde \calb_n \right) \\
		&\leq& \pp\left( \calb_1^{n-\ell-d+1}, \calb_{n-\ell+1}^{n-1}, \widetilde \calb_n \right) \\
		&=& \pp\left( \calb_{1}^{n-\ell-d+1}\right) \pp\left(\calb_{n-\ell+1}^{n-1}, \widetilde \calb_n \right) \\
		&=& a_{n-d-\ell+1} \beta_{\ell-1},
		\feq
		where we utilize the independence property of $\calb_1^{n-\ell-d+1}$ and $\calb_j$ with $j\geq n-\ell+1$ for the third line. The incorporation of this expression into equation \eqref{main_eq} establishes the lower bound.\\
		
		In order to obtain the upper bound, we further assume $n> \ell+2d$. Then, we can express $\beta_{n-1}$ as follows:
		\beqn
		\beta_{n-1} &=& \pp\left( \calb_{1}^{n-\ell}, \calb_{n-\ell+1}^{n-1}, \widetilde \calb_n \right) \notag \\
		&=& \pp\left( \calb_{1}^{n-\ell-d+1}, \calb_{n-\ell+1}^{n-1}, \widetilde \calb_n \right) \label{l_in_t1} \\
		&& - \pp\left( \calb_{1}^{n-\ell-d+1},  \widetilde{\calb_{n-\ell-d+2}^{n-\ell}}, \calb_{n-\ell+1}^{n-1}, \widetilde \calb_n \right). \label{l_in_t2}
		\feqn
		We note that \eqref{l_in_t1} can be further expressed as
		\beqn
		\pp\left( \calb_{1}^{n-\ell-d+1}, \calb_{n-\ell+1}^{n-1}, \widetilde \calb_n \right) &=& \pp\left( \calb_{1}^{n-\ell-d+1} \right) \pp\left(\calb_{n-\ell+1}^{n-1},
		\widetilde \calb_n \right) \notag \\
		&=& a_{n-\ell-d+1}\beta_{\ell-1}. \label{l_in_t4}
		\feqn    
		For the term \eqref{l_in_t2}, we observe that
		\beqn
		&& \pp\left( \calb_{1}^{n-\ell-d+1}, \widetilde{\calb_{n-\ell-d+2}^{n-\ell}}, \calb_{n-\ell+1}, \widetilde \calb_n \right)\notag\\
		&& \quad \leq \pp\left( \calb_{1}^{n-\ell-2d+2}, \widetilde{\calb_{n-\ell-d+2}^{n-\ell}}, \calb_{n-\ell+1}, \widetilde \calb_n \right) \notag\\&& \quad = \pp\left( \calb_{1}^{n-\ell-2d+2}\right) \pp\left(\widetilde{\calb_{n-\ell-d+2}^{n-\ell}} \right) \pp\left(\calb_{n-\ell+1},
		\widetilde \calb_n, \widetilde{\calb_{n-\ell-d+2}^{n-\ell}} \right) \notag\\
		&& \quad \leq  \frac{d-1}{d!} \pp\left( \calb_{1}^{n-\ell-2d+2} \right) = \frac{d-1}{d!} a_{n-\ell-2d+2}, \label{l_in_t3}
		\feqn
		where for the last inequality we applied 
		\beq
		\pp\left(\widetilde{\calb_{n-\ell-d+2}^{n-\ell}} \right) = \pp\left(\cup_{i=n-\ell-d+2}^{n-\ell} \widetilde \calb_i \right) \leq
		\sum_{i=n-\ell-d+2}^{n-\ell} \pp(\widetilde \calb_i) = \frac{d-1}{d!}.
		\feq
		We utilize the inequalities \eqref{l_in_t3} and \eqref{l_in_t4} to arrive at the following inequality:
		\beq
		\beta_{n-1} \geq \beta_{\ell-1} a_{n-\ell-d+1} - \frac{d-1}{d!} a_{n-\ell-2d+2}.
		\feq
		By inserting the above inequality back into \eqref{main_eq}, we have successfully completed the proof of the upper bound.
	\end{proof}
	
	Next, we will derive an additional bound by presenting the
	
	\begin{proof}[Proof of Theorem~\ref{growth_theorem}-(a)]
		
		We fix $\alpha \in \mathbb{N}$ and decompose the sequence $(Z_n)$ into equal blocks of length $\alpha d$. In what follows, we repeatedly utilize the stationarity of the sequence $(\calb_n)_{n\in \nn}$ and the independence of the blocks $\calb_n$ and $(\calb_{n+d+i})_{i\in \nn_0}$, for all $n\in \nn$. Let $\lambda := 1+d(\alpha-1)$. We first express
		\beqn
		&&  \left|\pp(\calb_1^{\alpha kd}) - \pp(\calb_1^{\lambda})^{k} \right| = \notag \\
		&& = \left| \sum_{j=0}^{k-1} \pp(\calb_1^{\lambda})^j \left(\pp(\calb_1^{\alpha d(k-j)}) - \pp(\calb_1^{\lambda}) \pp(\calb_1^{\alpha
			d(k-j-1)})\right) \right| \notag \\
		&& = \sum_{j=0}^{k-2} \pp(\calb_1^{\lambda})^j \left|\pp(\calb_1^{\alpha d(k-j)}) - \pp(\calb_1^{\alpha d(k-j-1)}, \calb_{\alpha d(k-j) -
			\lambda + 1 }^{\alpha d(k-j)})\right| \label{sum_2_T} \\
		&& \quad + \pp(\calb_1^{\lambda})^{k-1} \left| \pp(\calb_1^{\alpha d}) - \pp(\calb_1^{\lambda})\right|, \label{sum_3_T}
		\feqn
		where \eqref{sum_2_T} is obtained by using the independence of $\calb_1^{\alpha d(k-j-1)}$ and $\calb_{\alpha d(k-j)-\lambda+1}^{\alpha d(k-j)+1}$ for each $0\leq j \leq k-2$ and the stationarity.
		
		The above expression in equation \eqref{sum_2_T} can be bounded as follows:
		\beq
		&& \left|\pp(\calb_1^{\alpha d(k-j)}) - \pp(\calb_1^{\alpha d(k-j-1)}, \calb_{\alpha d(k-j)-\lambda +1}^{\alpha d(k-j)}) \right| \\&& \quad = \pp\left( \calb_1^{\alpha d(k-j-1)}, \widetilde{\calb_{\alpha d(k-j-1)+1}^{\alpha d(k-j)-\lambda}}, \calb_{\alpha
			d(k-j)-\lambda +1}^{\alpha d(k-j)}\right) \label{T_com} \\
		&& \quad \leq \pp\left( \cap_{s=0}^{k-j-2}\{ \calb_{\alpha ds+d}^{\alpha ds+\alpha d} \},\widetilde{\calb_{\alpha
				d(k-j-1)+1}^{\alpha d(k-j)-\lambda }}, \calb_{\alpha d(k-j)-\lambda+1}^{\alpha d(k-j)} \right) \\
		&& \quad \leq \pp\left( \cap_{s=0}^{k-j-2}\{ \calb_{\alpha ds+d}^{\alpha ds+\alpha d} \}, \calb_{\alpha d(k-j)-\beta+1}^{\alpha
			d(k-j)} \right) \\
		&& \quad = \pp\left( \calb_1^{\lambda} \right)^{k-j-1} \pp \left( \calb_{\alpha d (k-j)-\lambda+1}^{\alpha d (k-j)} \right) \\
		&& \quad = \pp\left( \calb_1^{\lambda} \right)^{k-j}.
		\feq
		This shows that the summation in equation \eqref{sum_2_T} is bounded by $k\pp(\calb_1^{\lambda})^{k}.$
		The rightmost factor of equation \eqref{sum_3_T} can be bounded above by two. Combining these bounds yields
		\beq
		\left|\pp(\calb_1^{\alpha kd}) - \pp(\calb_1^{\lambda})^{k}\right| &\leq&  k \pp(\calb_1^{\lambda})^{k} +  2\pp(\calb_1^{\lambda})^{k-1}  \\    &\leq& (k+2) \pp(\calb_1^{\lambda})^{k-1}.
		\feq
		This establishes the validity of \eqref{mineq}. Finally, as the limit $(a_n^{(1)})^{\frac{1}{n}}$ exists as $n$ approaches infinity, we can pass the limit to a subsequence and obtain
		\beq
		&&  \rho_v := \lim_n (a^{(1)}_n(v))^{\frac{1}{n}}  = \lim_{k} \pp(\calb_1^{\alpha k d})^{\frac{1}{\alpha k d}} \\
		&& \quad \leq \lim_k \left (k+2 + \pp(\calb_1^{\lambda}) \right)^{\frac{1}{\alpha k d}} \pp(\calb_1^{\lambda})^{\frac{k-2}{\alpha k
				d}}  \\
		&& \quad = a^{(1)}_{\lambda}(v)^{\frac{1}{\alpha d}} = \left( \frac{g^v_{0}(\cals_{\alpha d})}{(\alpha
			d)!}\right)^{\frac{1}{\alpha d}}.
		\feq
		This concludes the proof of the inequality \eqref{rho_bound}.
	\end{proof}
	
	Lastly, we present another set of bounds on the growth rate by providing the
	
	\begin{proof}[Proof of Theorem~\ref{growth_theorem}-(b)]
		The first part of this result follows directly from Theorem~\ref{war-thm}. However, to establish the lower bound, we must determine a lower bound for the dominant root of the equation \eqref{poly1}. To this end, we define $m=d+\ell-1$ and consider the function $f(x)=x^{m}-x^{m-1}-\beta_{\ell-1}$. The roots of the derivative of $f$ are $0$ with multiplicity $m-2$ and $\frac{m-1}{m}$ with multiplicity $1$. Since $\beta_{\ell-1}>0$, $0$ cannot be a root of $f$ and we have
		\beq
		f\left( \frac{m-1}{m} \right) = -\frac{(m-1)^{m-1}}{m^m} + \beta_{\ell-1}. 
		\feq
		Therefore, $\frac{m-1}{m}$ is a root of $f$ with multiplicity $2$ if and only if $\beta_{\ell-1}= \frac{(m-1)^{m-1}}{m^m}$.
		
		Since $\ell$ is chosen such that $\beta_{\ell-1} \neq \frac{(m-1)^{m-1}}{m^m}$, then $f(.)$ has $m$ roots $\lambda_1,\cdots \lambda_{m}$ each with
		multiplicity one. Thus, by \eqref{rho_bound} we can find constants $q_i$s independent of $n$ so that $a_n^{(1)}(v)$ is
		bounded below by
		\beq
		a_n^{(1)}(v) \geq \sum_{i=1}^{m} q_i \lambda_i^n.
		\feq
		We refer to the dominant root by $\rho_{v,\ell}$. Suppose $g(x)=x^{m}-x^{m-1}$ and observe that $f(x)-g(x)=\beta_{\ell-1}$ for
		all $x\in \cc.$ We apply Rouch\'e's Theorem (See \cite{silver1}, page 262) to locate the roots of $f(x)$ by understanding the roots of
		$g(x).$ 
		
		We need to find a simple closed contour, denoted by $C$, such that the inequality 
		\beqn \label{rouch_lab}
		|g(x)| > |f(x)-g(x)|
		\feqn
		holds for all $x$ on $C$. 
		According to  Rouch\'e's Theorem, if such a contour $C$ exists, then $f(x)$ and $g(x)$ must have the same number of roots (counting multiplicities) inside $C$. To satisfy the condition in equation \eqref{rouch_lab} on the boundary of $C$, we choose $C$ to be a circle with an appropriate radius $\delta$ centered at the origin. We need to ensure that $\delta<1$ so that the following inequality holds true:
		\beq
		|g(x)| &=& |x^m - x^{m-1}| \geq | |x|^m - |x|^{m-1}| = | \delta^m - \delta^{m-1} | \\
		&=&  \delta^{m-1} - \delta^m > \beta_{\ell-1} = |f(x) - g(x)|.
		\feq
		We can find such $\delta<1$ by choosing $\delta = \left(\beta_{\ell-1}\right)^{\frac{1}{(m-1)(1+\alpha)}}$ for some $\alpha>0$. Substituting this into the inequality above, we get:
		\beq
		1 < \left( 1 - \left(\beta_{\ell-1}\right)^{\frac{1}{(m-1)(1+\alpha)}}\right) \left( \frac{1}{\beta_{\ell-1}}\right)^{\frac{\alpha}{1+\alpha}}.
		\feq
		A straightforward calculation shows that there exists $\gamma>1$ such that
		\beq
		1-\beta_{\ell-1} - \left(\beta_{\ell-1}\right)^{\gamma} \leq \delta = \left(\beta_{\ell-1}\right)^{\frac{1}{(m-1)(1+\alpha)}}.
		\feq
		Since $0$ is a root of $g(x)$ with a multiplicity of $m-1$, $f(x)$ also has $m-1$ roots inside $C$, and the remaining root is not inside $C$. This root must be real, since otherwise, its conjugate would also be a root of $f(x)$ outside of $C$, which would be a contradiction. Moreover, this root is also the dominant root $\rho_{v,l}$ since it is farthest from the origin. Therefore, we have $1-\beta_{\ell-1} - \beta_{\ell-1}^{\gamma} \leq \delta\leq \rho_{v,l}$, as intended.
	\end{proof}
	
	\section{Consecutive containment in words} \label{word_sec}
	
	We will now begin our analysis concerning consecutive containment in words. At the outset, we leverage analogous argument with the same essence as of those in the previous section. To that goal, fix the integer number $\malp>d$, and define an i.i.d sequence of experiments $W:=(W_n)_{n\in \nn}$ with each $W_n$ sampled independently and uniformly from $[\malp].$  We will denote by $\pp$ the probability law in a suitable probability
	space that supports all random variables considered in this section. Consider a discrete-time Markov process $(Z_n)_{n\in \nn}$ on the space $\alp{d}$, where $Z_n:=W_{n}\cdots W_{n+d-1}$. The transition probability of the process $Z_n$ is defined as
	\beq
	P(w, \cala) = \frac{1}{\malp}\#\{y\ | w_2\cdots w_d y\in \cala\}, \quad w:=w_1\cdots w_d\in \alp{d}, \quad \cala\subset \alp{d}.
	\feq
	Recall \eqref{a_h_def}. Since $W_1 \cdots W_n$ is a uniformly randomly sampled word from $\alp n$, then
	\beq
	h_{\malp,n}^{(r+1)}(v)=\frac{g_{r}^v(\alp{n+d-1})}{\malp^{n+d-1}} = \pp(\con_v(W_1\cdots W_{n+d-1})=r).
	\feq
	In the rest of this section, we will use $\cald_i(v)$ to denote the event $\{\red(Z_i)\neq v\}$. We define $\cald_{i}^j(v)$ as $\cap_{s=i}^j \cald_s(v)$ for $i\leq j$. When $i>j$, we define $\cald_{i}^j(v)$ as the entire set. We will drop the subscript $v$ whenever the context is clear. 
	
	We now present the 
	
	\begin{proof} [Proof of Theorem~\ref{non-overlapping-bound-thm}-(b)] The proof shares many similarities with that of Theorem~\ref{non-overlapping-bound-thm}-(a), but also includes several differences. Therefore, our focus will be on highlighting these differences while omitting the parts that are similar. Since $v$ is non-overlapping, we can write
		\beq
		h_{\malp,n}^{(1)}(v) = h_{\malp,n-1}^{(1)}(v) - \sum_{k=0}^{m} (-1)^{k}  \HH_{\malp, k}(v) h^{(1)}_{\malp,n-kd+k-d}(v),
		\feq
		where $\HH_{\malp, k}(v) := \pp(\widetilde \cald_1, \cdots, \widetilde \cald_{k(d-1)+1})$ and $m:=\ceil{\frac{n-1}{d-1}}$. To compute $\HH_{\malp, k}(v)$ we proceed as follows. Let $$w:=(w_1\cdots, w_d, \cdots, w_{2d-1}, \cdots,
		w_{kd - (k-1)}, \cdots, w_{d(k+1)-k})$$ be a word in $\alp{(k+1)d-k}$ sampled by the random
		sequence $W_1\cdots W_{d(k+1)-k}$ satisfying $\widetilde \cald_1\cap \cdots \cap \widetilde \cald_{k(d-1)+1}$. Since the number of choices for $w$ equals to $\malp^{(k+1)d-k} \HH_{\malp, k}(v),$ we obtain the formula for $\HH_k(v)$ through an enumeration of $w$. Set $x_i:=w_{id+1}$ and assume without loss of generality $v_1<v_d$. The
		assumption $\widetilde \cald_1\cap \cdots \cap \widetilde \cald_{k(d-1)+1}$ dictates the following conditions on $(x_i)_{0\leq i \leq k+1}$ values: \par
		
		\begin{itemize}
			\item We must have $x_{0}<x_{1}<\cdots<x_{k+1}$. This immediately implies that $\HH_k(v)=0$ for all $k > \malp$.
			\item Out of $(k+1)d-k$ elements in the vector $w$, at least $v_1-1$ of them are smaller than $x_0$ and at most $(d-v_1)(k+1)$ of them are larger than $x_0$.  Hence,
			\beq
			v_1\leq x_0 \leq \malp -(d-v_1)(k+1).
			\feq
			\item Knowning $x_0$, the value of $x_{k+1}$ ranges between $x_0+(k+1)(v_d-1)$ and $\malp-(d-v_d)=\malp+v_d-d$. 
			\item Knowning $x_0$ and $x_{k+1}$, for any fixed $1\leq i \leq k$, the value of $x_i$ ranges between $x_{i-1}+v_d-v_1$ and $x_{k+1} - (k+1-i)(d-v_1)$.
		\end{itemize}
		Next, we enumerate $w$ by parametrizing with respect to $x_{0}<x_{1}<\cdots<x_{k}<x_{k+1}$. To that end, for any $1\leq i\leq k+1$, let us define $i$-th blocks of $w$ as $w_{(i-1)d+2}\cdots w_{id}$. For a fixed $1\leq i \leq k+1,$ we are interested in finding the number of choices for the $i$-th block when $x_i=w_{id+1}$ and $x_{i-1}=w_{(i-1)d+1}$ values are given. Observe that
		\begin{itemize}
			\item There are exactly $v_d-v_1-1$ values in the $i$-th block that are between $x_{i-1}$ and $x_i$. The number possible choices for $i$-th block is therefore
			\beq
			\binom{x_{i}-x_{i-1}-1}{v_d-v_1-1}.
			\feq
			\item There are only $(\malp-x_{i-1})-(k-i)(d-v_1-1)-1$ values that can be assigned to $d-v_d-1$ elements of the $i$-th block that are larger than $x_{i}=z_{id+1}$. The number choices are
			\beq
			\binom{(\malp-x_{i-1})-(k-i)(d-v_1-1)-1}{d-v_d-1}.
			\feq
			\item Similarly, there are only $x_i-1-(i-1)(v_d-2)-(i-1)=x_{i-1}-(v_d-1)i-1$ values that can be assigned to $v_1-1$ elements of the $i$-th block that are smaller than $x_{i-1}=z_{(i-1)d+1}.$ The number of choices are 
			\beq
			\binom{x_{i-1}-(v_d-1)i-1}{v_1-1}.
			\feq
		\end{itemize}
		By collecting all the observations we have made so far, using the convention that $\sum_{i}^j . =0$ for $j<i$, we arrive at the following expression:
		\beqn \label{Ln_value_word}
		&& \malp^{(k+1)d-k} \HH_{\malp, k}(v) = \sum_{x_0=v_1}^{(k+1)v_1-k}\ \  \sum_{x_{k+1}=x_0+(k+1)(v_d-1)}^{k(d-1)+v_d} \ \  \sum_{x_1=x_0+v_d-v_1}^{x_{k+1}-k(d-v_1)} \cdots \sum_{x_k=x_{k-1}+v_d-v_1}^{x_{k+1}-(d-v_1)} \\
		&& \qquad \quad \prod_{i=1}^{k+1}\binom{x_{i} -
			x_{i-1}-1}{v_d-v_1-1}  \binom{(\malp-x_{i-1})-(k-i)(d-v_1-1)-1}{d-v_d-1} \binom{x_{i-1}-(v_d-1)i-1}{v_1-1}. \notag
		\feqn
		Therefore, \eqref{www1} together with \eqref{Ln_value_word} completes the proof.
	\end{proof}
	
	\subsection{Connection with a classical probability problem} \label{li_section}
	
	In this section, our focus is on proving Theorem \ref{words_thm}. To achieve this goal, we will draw a connection to a classical probability problem that has been extensively studied. However, before we do so, we need to define the framework and review some relevant existing results. In \cite{li1}, Li proposed a martingale method for studying the occurrence of a set of predefined sequences of observations (patterns) in a sequence of independently repeated experiments. This method involves describing a gambling game among multiple teams and forming a system of equations that relate the expected waiting time until any consecutive observation sequence appears to the probability that one of the patterns precedes the remaining ones. Gerber and Li \cite{li2} extended this work by using a Markov chain embedding technique to calculate the generating functions of waiting time distributions for patterns. The martingale method was further refined by Pozdnyakov et al. \cite{steel1} and by Zajkowski \cite{kr1}. It should be noted that the concept of a pattern in these papers differs slightly from our definition of a pattern in this paper; Li's ``pattern" refers to what we refer to as a ``pattern instance." Here, we present the framework in the style used in \cite{kr1}. We remark that the original framework is more general; however, for simplicity, we simplify the framework to suit our purposes.
	
	To establish context, let us define a correlation value between two words of equal length that is crucial to our analysis. Specifically, for any $\alpha$ in the interval $(0,1)$ and any two words $u$ and $s$ in $\alp d$, the $\alpha$-correlation of $u$ and $s$ is defined as:
	
	\beqn \label{cor_def}
	(u*s)_{\alpha, \malp} = \sum_{i=1}^d \left(\frac{\malp}{\alpha}\right)^i \one{u[d-i+1,d] = s[1,i]},
	\feqn
	where we use the notation $s[i,j]:=s_i\cdots s_j$ to denote any substring of the word $s:=s_1\cdots s_d$, where $1\leq i\leq j\leq d$. 
	
	Suppose $\cali$ be a set of words in $\alp d$. We define various notions of hitting times for the set $\cali$. Throughout our proofs, we may use different values for $\cali$ and may adopt different subscripts or postscripts in the notations to maintain focus on the context or simplify the notation. For any two words $t$ and $s$, we define
	\beq
	\tau(t,s) := \inf \{ n\in \nn \ | \  s \mbox{ suffixes the sequence } t_1\cdots t_d W_{1}\cdots W_n \}
	\feq
	to be the first time that the sequence $u_1\cdots u_\ell W_{1}\cdots W_n$ ends with the word $s$ as a postfix. Then, for any given word $t$, we define     
	\beq
	T^{(1)}(t) := \inf_{s\in \cali} \tau(t, s).
	\feq
	At $T^{(1)}(.)$, the visited word can be any of the words in $\mathcal{I}$. Thus, we introduce $Q(.,s)$ to denote the probability that $s$ is the specific instance. More formally, we set $Q(t,s)$ as the probability that $T^{(1)}(t) = \tau(t, s)$, that is:
	\beq
	Q(t,s):= \pp(T^{(1)}(t) = \tau(t, s)).
	\feq
	We will now explain how to obtain information about the hitting time and hitting probability. For this purpose, we will follow the framework of \cite{kr1}, but interested readers can find more detailed discussions in \cite{li2, li1, poz1, steel1}. Recall the definition of $\alpha$-correlation from \eqref{cor_def}. Using Proposition 4.1 from \cite{kr1}, we obtain the following equation:
	\beqn
	\ee(\alpha^{T^{(1)}(t)}) + \sum_{s\in \cali} (s*q)_\alpha \ee\left(\alpha^{T^{(1)}(t)} \one{T^{(1)}(t) = \tau(t,s)}\right) = (t*q)_{\alpha, k}, \qquad \forall q\in \cali. \label{T_1_system}
	\feqn
	We also have the trivial equation
	\beqn \label{T_1_system_2}
	\ee(\alpha^{T^{(1)}(t)}) = \sum_{s\in \cali_v(k)} \ee\left(\alpha^{T^{(1)}(t)}\one{T^{(1)}(t) = \tau(t, s)}\right).
	\feqn    
	To express the solution of the system \eqref{T_1_system} and \eqref{T_1_system_2}, we construct a correlation matrix $R(\alpha)$ for $\mathcal{I}$, which contains $\alpha$-correlation values among the words in $\cali$. For any pair $(u,s)\in \cali \times \cali$, the $(u,s)$-entry of the matrix, denoted by $[R(\alpha)]_{u, s}$, is defined as follows:
	\beqn \label{R_def}
	[R(\alpha)]_{u, s} := (u*s)_{\alpha, k}.
	\feqn
	
	Suppose $u,s\in \cali$ and $t\in \alp d$. We define $R_{u,t}(\alpha)$ (resp. $R^{u}(\alpha)$) to be the matrix attained by replacing the column (resp. the row) corresponding to $u$ in $R(\alpha)$ with the column (resp. the row vector) $((t*u)_\alpha)_{s\in \cali}$ (resp. $(1)_{s\in \cali}$). $R_{s,t}^u(\alpha)$ is obtained by replacing the $u$-th column of $R_{s,t}(\alpha)$ with the column vector of units $(1)_{u\in \cali}$. The solution of the system \eqref{T_1_system} and \eqref{T_1_system_2} can be described in terms of $R(\alpha)$ as defined by \eqref{R_def}, $R^u(\alpha)$, and $R_{u,t}(\alpha)$. More specifically,  Proposition 4.1, \cite{kr1} states 
	\beqn \label{alpha_T_gen_form}
	\ee(\alpha^{T^{(1)}(t)}) = 1 - (1-\alpha) \frac{\det R(\alpha) - \sum_{u\in \cali} \det R_{u,t}(\alpha)}{(1-\alpha)\det
		R(\alpha) + \sum_{u\in \cali} \det R^u(\alpha)},
	\feqn
	and
	\beqn \label{alpha_T_u_gen_form}
	\ee\left(\alpha^{T^{(1)}(t)} \one{T^{(1)}(t) = \tau(t, s)} \right) = \frac{(1-\alpha)\det R^{s}(\alpha) + \sum_{u\in \cali} \det
		R_{s,t}^u(\alpha) }{(1-\alpha)\det R(\alpha) + \sum_{u\in \cali} \det R^u(\alpha)}.
	\feqn
	As discussed in Remark 4.2 of \cite{kr1}, we can greatly simplify these solutions by substituting $t=\epsilon$ and using the identity $(\epsilon*s)_\alpha=0$. Let us define $\tau(u)$ to be $\tau(\epsilon, u)$ and set $T^{(1)}:=T^{(1)}(\epsilon)$. Then, 
	\beq
	\ee(\alpha^{T^{(1)}}) = 1- (1-\alpha) \frac{\det R(\alpha)}{(1-\alpha)\det R(\alpha) + \sum_{u\in \cali} \det
		R^u(\alpha)}
	\feq
	and 
	\beq
	\ee\left(\alpha^{T^{(1)}} \one{T^{(1)} = \tau(u)} \right) = \frac{\det R^u(\alpha) }{(1-\alpha)\det R(\alpha) + \sum_{s\in \cali}
		\det R^s(\alpha)}.
	\feq
	Next, we define the sequences of random variables $(T^{(r)})_{r\geq 1}$ and $(T^{(r)}(u))_{r\geq 1}$ as follows. For $r\geq 2$, we set
	\beqn \label{T_I_def}
	T^{(r)} := \inf \{ n > T^{(r-1)} \ | \ W_{n-d+1}\cdots W_n\in \cali \},
	\feqn 
	and
	\beq
	T^{(r)}(u) := \inf \{ n > T^{(r-1)}(u) \ | \ W_{n-d+1}\cdots W_n \in \cali \}.
	\feq
	A straightforward calculation yields the following lemma, which is the main ingredient in the proof of Theorem \ref{words_thm}.
	
	\begin{lemma} \label{lemma_Ti}
		For any $r\in\nn,$ $\ee(\alpha^{T^{(r)}})$ depends only on the matrix $R(\alpha)$s, $R^u(\alpha)$s,
		$R_{s,t}(\alpha)$s, and $R_{s,t}^u(\alpha)$s.
	\end{lemma}
	\begin{proof}
		Suppose $r\geq 2$. We use the definition \eqref{T_I_def} to write   
		\beq
		\ee(\alpha^{T^{(r)}(s)})  &=& \sum_{u\in \cali} \ee\left(\alpha^{T^{(r)}(s)}\one{T^{(1)}(s)=\tau(s,u)}\right)  \notag \\
		&=& \sum_{u\in \cali} \ee\left(\alpha^{\tau(s, u)+T^{(r-1)}(u)}\one{T^{(1)}(s)=\tau(s,u)}\right)  \notag \\
		&=& \sum_{u\in \cali}\ee\left(\alpha^{T^{(1)}(s)} \one{T^{(1)}(s) = \tau(s,u))} \right) \ee\left(\alpha^{T^{(r-1)}(u)}\right).
		\feq
		An inductive argument, based on this recursive relation, along with \eqref{alpha_T_gen_form} and \eqref{alpha_T_u_gen_form} for $T^{(r-1)}(s)$, provides us with the proof. However, we will not provide an explicit description of them since they are not necessary for our purposes.
	\end{proof}
	
	We now have the necessary tools to prove Theorem~\ref{words_thm}. To this end, we will briefly discuss the correlation matrix of a pattern and how it can aid in computing the generating function of $g_r^v(\alp{n})$, denoted by $\wmf_{r,\malp}^v$:
	\beq
	&&\wmf_{r,\malp}^v(x) = \sum_{n=d}^\infty
	g_r^v(\alp{n}) x^n, \quad \malp \in \nn, \ r\in \nn_0.
	\feq
	Take $\malp \geq d$ fixed, and consider any pattern $v\in \cals_d$. For any subset $\cala \subseteq \alp d$, let $v_\cala$ be the instance of the pattern $v$ formed by using only the letters in $\cala$. For example, $231_{\{1,3,9\}}=391$. Define $\cali_v(\malp)$ to be the set of all instances of $v$ in the alphabet $[\malp]$, that is,
	\beqn \label{I_kv_def}
	\cali_v(\malp):=\{ v_\cala \ | \cala \in [\malp]_d \}.
	\feqn
	For instance, if $v=132$ and $k=4$, then 
	\beq
	\cali_v(\malp) = \{132,142,143,243\}.
	\feq
	We create a correlation matrix, denoted by $R(\alpha, v, \malp)$, for $\cali_v(k)$, which is composed of $\alpha$-correlation values among the words in $\cali_v(\malp)$. For any pair $(u,s)\in \cali_v(\malp) \times \cali_v(\malp)$, the entry of the matrix corresponding to $(u,s)$, denoted by $[R(\alpha, v, \malp)]_{u, s}$, is defined as $(u*s)_{\alpha, \malp}$.
	
	Recall the definition of $\cali_v(\malp)$ from \eqref{I_kv_def}. To simplify the notation, we replace the superscript and subscript $u$, $s$, $t$ in the correlation matrices by the corresponding $d$-sets $\cala$, $\calb$, and so on. Thus, when we use $\cala$ as the subscript or superscript, it should be replaced by $v_\cala$. For a fixed pattern $v\in \cals_d$ and alphabet size $\malp$, the matrix $R(\alpha,v,\malp)$ is a $\binom{\malp}{d}\times \binom{\malp}{d}$ matrix, and for any pair $(\cala,\calb)\in [\malp]_d \times [\malp]_d$, the $(\cala,\calb)$-entry of $R(\alpha,v,\malp)$, denoted by $[R(\alpha,v,\malp)]_{\cala, \calb}$, can be expressed as follows:
	\beq
	[R(\alpha,v,\malp)]_{\cala, \calb} &=& (v_\cala*v_\calb)_{\alpha, k} = \frac{1}{\alpha^d} \sum_{i=1}^d \malp^i \alpha^{d-i} \one{v_\cala[d-i+1, d]
		= v_\calb[1, i]} \notag \\
	&=& \one{\cala=\calb}\left(\frac{\malp}{\alpha}\right)^d + \sum_{i\in \calo_v} \left(\frac{\malp}{\alpha} \right)^i \one{v_\cala[d-i+1, d]
		= v_\calb[1, i]}.
	\feq
	
	For example, if $v\in \cals_d$ is a non-overlapping pattern, then for $\malp\geq d$, the $(\cala,\calb)$-entry of the correlation matrix $R(\alpha, v, \malp)$ is given by \eqref{R_nonover}:
	\beqn \label{R_nonover}
	[R(\alpha, v, \malp)]_{\cala,\calb} = \begin{cases}
		\left(\frac{\malp}{\alpha}\right)^d  & \text{if } 	\cala=\calb \\
		\frac{\malp}{\alpha} & \text{if }  (v_\cala)_d = (v_\calb)_1 \\
		0 & \text{otherwise}
	\end{cases}.
	\feqn
	Note that for $(\cala,\calb)$ to satisfy the condition $(v_\cala)_d = (v_\calb)_1$, the $v_d$-th smallest element of the set $\cala$ must be the $v_1$-th smallest element of the set $\calb$.
	
	Recall equation \eqref{T_I_def}. We substitute $\cali$ with $\cali_v(\malp)$ in \eqref{T_I_def} and denote the resulting functions and values by $T_v^{(1)}(.)$ and $T_v^{(1)}$, respectively. Using this notation, we can state that the equation $\red(W_{n-d+1}\cdots W_n) = v$ is equivalent to $W_{n-d+1}\cdots W_n \in \cali_v(\malp)$. Moreover, for any $r\in \nn_0,$ we have
	\beq
	\cup_{j=0}^r \{ \con_v(W_1\cdots W_n)=j \} = \cup_{j=0}^r \{ W_1\cdots W_n \in G_j^v(\alp n) \} = \{ T^{(r+1)}_v > n \}.    
	\feq
	
	Let $f_r^v(\alp n)$ denote the number of words in $\alp n$ containing at most $r$ order-isomorphic copies of $v$. We can express this number as the sum of $g_j^v(\alp n)$. Hence, we will write
	\beq
	f_r^v(\alp n) = \sum_{j=0}^r g_j^v(\alp n) = \malp^n \sum_{j=0}^r \pp(W_1\cdots W_n \in G_j^v(\alp n)) = \malp^n
	\pp(T^{(r+1)}_v>n).
	\feq
	
	Lastly, we present the 
	
	\begin{proof}[Proof of Theorem \ref{words_thm}] 
		Let us begin by recalling that it is well known that if $\malp \geq d$, then $T_{v}^{(r)}<\infty$ almost surely and hence $\ee(\alpha^{T_{v}^{(r)}})<\infty$. This is due to the fact that $(W_{n-d+1}\cdots W_n)_{n\geq d}$ is an irreducible finite Markov chain on $\alp d$, which means that the hitting time of any nonempty subset of $\alp d$ has an exponential tail and is almost surely finite.
		
		Now, let us inspect the generating functions of the random variables $T_{r}(v,k)$. To that end, observe that Fubini's theorem can be utilized to write
		
		\beq
		\ee(\alpha^{T_{v}^{(r)}}) &=& \int_0^1 \pp(\alpha^{T_{v}^{(r)}}\geq x) dx = \sum_{\ell=0}^\infty
		\int_{\alpha^{\ell+1}}^{\alpha^{\ell}} \pp(\alpha^{T_{v}^{(r)}}\geq x) dx \\
		&=& \sum_{\ell=0}^\infty \pp(\alpha^{T_{v}^{(r)}}\geq \alpha^{\ell+1}) \int_{\alpha^{\ell+1}}^{\alpha^{\ell}} dx =
		\sum_{\ell=0}^\infty \pp(\alpha^{T_{v}^{(r)}}\geq \alpha^{\ell+1}) (\alpha^\ell-\alpha^{\ell+1}) \\
		&=& \sum_{\ell=0}^\infty \pp(T_{v}^{(r)}\leq \ell+1) (\alpha^\ell-\alpha^{\ell+1}) = \sum_{\ell=0}^\infty \left( 1-
		\pp(T_{v}^{(r)} > \ell+1)\right) (\alpha^\ell-\alpha^{\ell+1}) \\
		&=& 1 - \frac{1-\alpha}{\alpha} \sum_{\ell=1}^\infty \pp(T_{v}^{(r)}> \ell)\alpha^{\ell} = 1 - \frac{1-\alpha}{\alpha}
		\sum_{\ell=1}^\infty f_{r-1}^v(\alp \ell) \left(\frac{\alpha}{\malp}\right)^\ell.
		\feq 
		Let us continue by noting that for $\malp\geq d$, we can express
		\beqn \label{words_alpha_tau_v}
		\wmf_{r,k}^v\left(\frac{\alpha}{\malp}\right) & =& \sum_{\ell=1}^\infty f_{r}^v(\alp \ell)
		\left(\frac{\alpha}{\malp}\right)^\ell- \sum_{\ell=1}^\infty f_{r-1}^v(\alp \ell) \left(\frac{\alpha}{\malp}\right)^\ell \notag
		\\
		&=& \frac{\alpha}{1-\alpha} \left(\ee(\alpha^{T_{v}^{(r)}}) - \ee(\alpha^{T_{v}^{(r+1)}})\right).
		\feqn 
		To complete the argument, we observe that Equation \eqref{khor-cond} implies that $R(\alpha,w,k)=R(\alpha, v,k)$ holds true, and this same fact also applies to all the modified matrices described in the paragraph following Equation \eqref{R_def}, for both the patterns $v$ and $w$. Therefore, by equation \eqref{words_alpha_tau_v} and Lemma \ref{lemma_Ti}, we have
		\beq
		\wmf_{r,\malp}^v\left(\frac{\alpha}{\malp}\right)  &=& \frac{\alpha}{1-\alpha} \left(\ee(\alpha^{T_{v}^{(r)}}) - \ee(\alpha^{T_{v}^{(r+1)}})\right) \\
		&=& \frac{\alpha}{1-\alpha} \left(\ee(\alpha^{T_{w}^{(r)}}) - \ee(\alpha^{T_{w}^{(r+1)}})\right) = \wmf_{r,\malp}^w\left(\frac{\alpha}{\malp}\right),
		\feq
		which completes the proof.
	\end{proof}
	
	\subsection{From words to permutations}
	
	This section is devoted to proving that for any two patterns $v$ and $w$ in $\cals_d$, \eqref{khor-cond} implies that $v$ and $w$ are strongly c-Wilf-equivalent in permutations. To accomplish this, we rely on a lemma that allows us to extend the result of Theorem \ref{words_thm} to permutations. This lemma was previously presented in \cite{reza1} within the context of classical pattern containment. However, we include it here for completeness, even though the proof is identical.
	
	\begin{lemma} \label{i^nSn} 
		Let $v$ be any pattern. Then, for $r\in \nn_0,$
		\beqn \label{f_perm_n}
		&&g_r^v(\cals_n) =  \sum_{\malp =1}^{n} (-1)^{n-\malp } \binom{n}{\malp } g_r^v(\alp n).
		\feqn
	\end{lemma} 	
	\begin{proof} [Proof of Theorem \ref{i^nSn}-(a)]
		We define $\caly_n(\cala)$ as the set of all words of length $n$ whose distinct set of letters is exactly $\cala$. Formally,
		\beq
		&&\caly_n(\cala) := \cala^n \setminus \left(\cup_{e\in \cala} \cala_e^n\right)\quad \mbox{where}\quad
		\cala_e:=\cala\setminus\{e\}.
		\feq	
		Recall equation \eqref{gfr} and express it using the notation introduced above as follows:
		\beqn  \label{LnA}
		\calg_r^v(\caly_n(\cala)) = \calg_r^v(\cala^n) \setminus \left(\cup_{e\in \cala} \calg_r^v(\cala_e^n)\right).
		\feqn
		Furthermore, for any non-empty subset $\calj\subset \cala$, we have
		
		\beq
		\cup_{e\in \calj} \calg_r^v(\cala_e^n) = \calg_r^v(\cup_{e\in \calj} \cala_e^n), \quad \mbox{and} \quad \cap_{e\in \calj}
		\calg_r^v(\cala_e^n) = \cap_{e\in \calj} \calg_r^v(\cala_e^n).
		\feq
		Using the inclusion-exclusion principle on equation \eqref{LnA}, we get:
		\beq
		&& g_r^v(\caly_n(\cala)) = g_r^v(\cala^n) - \sum_{\calj\subset \cala} (-1)^{n-\#\calj}g_r^v(\cap_{e\in \calj} \cala_e^n
		) \\
		&& = \sum_{\calj \subset \cala} (-1)^{\#\calj}g_r^v((\cala\setminus \calj)^n ).
		\feq
		Since $g_0^v(\cala^n) = g_0^v(\calb^n)$ for any $\cala,\calb\subset [n]$ where $\cala$ and $\calb$ have the same size, we have
		\beq
		g_r^v(\caly_n([a])) = g_r^v([a]^n) + \sum_{\malp =1}^{a-1} (-1)^{\malp } \binom{a}{\malp } g_r^v([a-\malp ]^n),
		\feq
		for each $a\in \nn$. Finally, note that $\caly_n([n]) = \cals_n$. This completes the proof of \eqref{f_perm_n}.
	\end{proof}
	
	An immediate consequence of this lemma is the following corollary.
	
	\begin{corollary} \label{word-perm-cor}
		(strongly) c-Wilf-equivalence in words implies (strongly) c-Wilf-equivalence in permutations.
	\end{corollary}
	
	The result follows from combining Theorem \ref{words_thm} with Corollary \ref{word-perm-cor}.
	
	\section*{Acknowledgement}
	
	I would like to express my gratitude to Sergi Elizalde for providing feedback on an initial draft, and to Ashwin Sah for bringing to my attention the proof of Elizalde's conjecture for non-overlapping patterns as presented in \cite{lee1}.

\end{document}